\documentclass[11pt,reqno]{amsart}
\usepackage{amsmath}
\usepackage{amsmath,amssymb,amsfonts,amsthm}
\usepackage{hyperref}
\usepackage{mathrsfs}
\usepackage{enumitem}
\usepackage{geometry}
\hypersetup{
	colorlinks=true,
	linkcolor=blue,
	citecolor=blue,
	urlcolor=blue
}

\newtheorem{theorem}{Theorem}[section]
\newtheorem{lemma}[theorem]{Lemma}

\newtheorem{corollary}[theorem]{Corollary}

\theoremstyle{definition}

\theoremstyle{remark}
\newtheorem{remark}[theorem]{Remark}

\title{A Pluripotential-Theoretic Approach to Finiteness of Polarized Calabi--Yau Manifolds}

\author{Truong Dinh Dat}

\address{Nguyen Trai University, Hanoi, Viet Nam}

\email{truongdinhdat14081994qb@gmail.com}

\begin{document}
	\begin{abstract}
		We propose a pluripotential-theoretic framework for studying
		finiteness questions for polarized Calabi--Yau manifolds.  The central
		idea is to use global \(m\)-Hessian capacities as a quantitative
		intermediate between complex potential theory and global geometric
		boundedness.  We formulate a chain of uniform estimates connecting
		\(m\)-Hessian capacity, volume--capacity inequalities,
		capacity--perimeter estimates, Sobolev
		bounds, and non-collapsing of the associated Ricci-flat Kähler
		metrics.  Under suitable regularity assumptions, these estimates lead
		to uniform diameter and curvature control and hence to metric
		compactness.
		
		We then describe the algebraic part of the argument.  Uniform
		projective embeddings with bounded degree lead, through the
		Macaulay--Gotzmann theory, to only finitely many possible Hilbert
		polynomials and hence to a finite collection of Hilbert schemes.
		Ehresmann's fibration theorem then converts boundedness of the smooth
		algebraic families into finiteness of diffeomorphism and topological
		types.
		
		The resulting framework isolates the main analytic and algebraic
		bottlenecks in a potential-theoretic approach to Calabi--Yau
		finiteness.  In particular, the capacity--perimeter estimate and the
		derivation of uniform algebraic embedding data from the preceding
		metric estimates constitute the essential steps requiring further
		development.  Thus the paper provides a structured program linking
		complex Hessian potential theory with geometric and algebraic
		boundedness, rather than assuming that these implications are
		automatic.
	\end{abstract}

	

	\maketitle
	
	%%%%%%%%%%%%%%%%%%%%%%%%%%%%%%%%%%%%%%%%%%%%%%%%%%%%%%
\section{Introduction}

A fundamental problem in the geometry of Calabi--Yau manifolds is to
understand to what extent geometric and analytic bounds can force
finiteness of the underlying algebraic or topological structures.
The classical theory of the Calabi conjecture provides canonical
Ricci-flat Kähler metrics on polarized Calabi--Yau manifolds
\cite{Calabi1954,Yau1977,Yau1978}. These metrics offer a natural bridge
between complex geometry, nonlinear partial differential equations,
and topology.

In this paper we propose a pluripotential-theoretic framework for
studying finiteness questions for polarized Calabi--Yau manifolds.
The central idea is to use global \(m\)-Hessian capacities as a
quantitative intermediate between complex potential theory and global
geometric boundedness. Schematically, our strategy is
\[
\begin{split}
	\text{uniform \(m\)-Hessian capacity estimates}
	&\Longrightarrow
	\text{uniform non-collapsing} \\
	&\Longrightarrow
	\text{diameter and regularity estimates
		for the associated Ricci-flat metrics}\\
	&\Longrightarrow
	\text{Metric compactness}\\
	&\Longrightarrow
	\text{Uniform algebraic boundedness} \\
	&\Longrightarrow
	\text{topological finiteness}.
\end{split}
\]

The potential-theoretic starting point is the global
\(m\)-Hessian capacity associated with a Kähler form \(\omega\):
\[
\operatorname{Cap}_{m,\omega}(E)
=
\sup\left\{
\int_E
(\omega+dd^c v)^m\wedge\omega^{n-m}
:\,
v\in SH_m(M,\omega),\ -1\leq v\leq0
\right\}.
\]
Complex Hessian equations and their pluripotential theory have been
developed extensively in the compact Kähler setting
\cite{Blocki2005,HouMaWu2009,DinewKolodziej2009,DinewLu2015}.
These works provide, in particular, fundamental a priori estimates,
mixed Hessian inequalities, and uniqueness results which form part of
the analytic background of our approach. For the broader
pluripotential-theoretic framework and its connections with complex
Monge--Ampère equations, we refer also to
\cite{DemaillyBook,KolodziejBook,GuedjZeriahiBook}.

Our first objective is to obtain uniform volume--capacity estimates
of the form
\[
\operatorname{Vol}_{\omega}(K)
\leq
C\,\operatorname{Cap}_{m,\omega}(K)^\alpha.
\]
Such inequalities are a central tool in pluripotential theory and
allow one to convert capacity estimates for sublevel sets of
\(m\)-subharmonic functions into volume estimates. They are also an
important ingredient in the derivation of uniform \(L^\infty\)
estimates for degenerate complex Hessian equations
\cite{DinewKolodziej2009,LuNguyen2015}.
The general pluripotential-theoretic background for the corresponding
Monge--Ampère theory can be found in
\cite{KolodziejBook,GuedjZeriahiBook,Guedj2007}.

The next step is geometric. A volume--capacity inequality alone does
not imply an isoperimetric inequality. The essential additional
ingredient is a suitable capacity--perimeter estimate. Once such an
estimate is available with appropriate exponents, it can be combined
with the volume--capacity inequality to obtain a uniform
isoperimetric inequality. The latter has a direct geometric
consequence: by the coarea formula applied to geodesic balls, one
obtains a uniform non-collapsing estimate
\[
\operatorname{Vol}_{\omega_j}
\bigl(B_{\omega_j}(x,r)\bigr)
\geq
\kappa r^{2n},
\qquad
0<r\leq r_0,
\]
with constants \(\kappa,r_0>0\) independent of \(j\).

For polarized Calabi--Yau manifolds, the canonical metrics are
Ricci-flat by Yau's theorem \cite{Yau1978}. Consequently, once
uniform non-collapsing and a uniform total-volume bound are available,
Bishop--Gromov comparison gives a uniform diameter estimate; see also
the general treatment of comparison geometry in
\cite{BishopCrittenden}. More
generally, the geometric estimates considered below are closely
related to the standard compactness theory for Riemannian manifolds
with curvature and injectivity-radius control
\cite{CheegerGromovTaylor1982,CheegerGromov1986,Besse1987}.

Further regularity estimates can, under suitable assumptions on the
background geometry and the Monge--Ampère data, lead to uniform
curvature bounds and hence to metric compactness. We emphasize,
however, that diameter and non-collapsing alone do not imply uniform
curvature bounds. Thus the required curvature estimates are treated
as genuine additional hypotheses in the corresponding part of our
framework. The interaction between metric limits and algebraic
geometry in the Kähler setting is also reflected in the work of
Donaldson and Sun \cite{DonaldsonSun2014,DonaldsonSun2016}.

There is also a substantial algebro-geometric background for
boundedness questions of polarized varieties. Classical constructions
in the theory of algebraic surfaces and projective embeddings already
illustrate how numerical polarization data constrain algebraic
families \cite{Mumford1966}. In the Kähler setting, related boundedness
questions for canonical metrics and polarized manifolds have also
played an important role in the work surrounding the Calabi
conjecture and Kähler--Einstein geometry
\cite{Tian1990,Tian1997}.

The final part of the argument is algebraic. Suppose that a fixed
positive power \(L^q\) of the polarization is uniformly very ample and
that the dimensions of the spaces
\[
H^0(M,L^q)
\]
are uniformly bounded. The resulting embeddings
\[
\Phi:M\hookrightarrow\mathbb P^N
\]
then have uniformly bounded degree, since
\[
\deg\Phi(M)
=
\int_Mc_1(L^q)^n
=
q^n\int_Mc_1(L)^n.
\]
The resulting bounded-degree projective families are governed by the
theory of Hilbert functions and Hilbert polynomials. In particular,
Macaulay's growth theorem and Gotzmann's regularity theorem imply the
finiteness of the possible Hilbert polynomials in fixed dimension,
ambient projective space, and bounded degree
\cite{Macaulay1927,Gotzmann1978}.
Standard references for the algebraic-geometric background include
\cite{Hartshorne1977,Eisenbud1995}, while the construction of Hilbert
schemes goes back to Grothendieck's foundational work
\cite{Grothendieck1961}.

It follows that, after passing to finitely many Hilbert polynomials,
the relevant smooth varieties occur as fibers of finitely many
universal families. On each connected component of the smooth locus,
Ehresmann's fibration theorem implies that the family is
\(C^\infty\)-locally trivial. Hence the smooth fibers have the same
diffeomorphism type on each connected component
\cite{Ehresmann1951}. This provides the final passage from algebraic
boundedness to topological finiteness.

The purpose of this paper is therefore not to regard any one of these
implications as automatic. Rather, we isolate a concrete chain of
uniform estimates through which pluripotential theory could lead to
finiteness for polarized Calabi--Yau manifolds. In particular, the
capacity--perimeter estimate and the passage from metric control to
uniform algebraic embedding data are identified as the principal
nonlinear and algebro-geometric bottlenecks. Establishing these
bridges in sufficiently general families would turn the framework
below into a genuine finiteness theorem.

We formulate the analytic, geometric, and algebraic steps separately
so that each hypothesis is explicit. This also allows the method to
be applied under different sources of uniformity, for example from
\(m\)-Hessian estimates, Ricci-flat metric estimates, or independent
algebraic boundedness results.

The paper is organized as follows. In Section~2 we recall the
\(m\)-Hessian capacity and establish the uniform volume--capacity
estimates. Section~3 develops the corresponding potential-theoretic
estimates for complex Hessian equations. In Section~4 we derives uniform
non-collapsing. Section~5 treats diameter and regularity estimates
for the associated Ricci-flat metrics. In section~6 we study metric compactness. Section~7 develops the
algebraic boundedness argument using projective embeddings,
Hilbert polynomials. Finally,
Section~8 proves the resulting topological finiteness statement.

	%%%%%%%%%%%%%%%%%%%%%%%%%%%%%%%%%%%%%%%%%%%%%%%%%%%%%%
	%%%%%%%%%%%%%%%%%%%%%%%%%%%%%%%%%%%%%%%%%%%%%%%%%%%%%%

\section{Preliminaries}

Throughout the paper, \(M\) denotes a compact Kähler manifold of
complex dimension \(n\), and \(\omega\) is a Kähler form on \(M\).
We use the normalization
\[
d^c=\frac{i}{2}(\bar\partial-\partial),
\qquad
dd^c=i\partial\bar\partial,
\]
although the precise normalization will not affect the arguments
below.

\subsection{Polarized Calabi--Yau manifolds}

A polarized Calabi--Yau manifold is a pair \((M,L)\), where \(M\) is a
smooth projective manifold of complex dimension \(n\) with
\[
c_1(M)=0
\]
and \(L\) is an ample line bundle on \(M\).  We write
\[
V_L:=\int_M c_1(L)^n
\]
for the polarization volume.

After replacing \(L\) by a suitable positive tensor power, we may
assume that \(L\) is very ample and hence determines a projective
embedding
\[
\Phi_L:M\hookrightarrow\mathbb P^N.
\]
The degree of the image is
\[
\deg \Phi_L(M)
=
\int_M c_1(L)^n.
\]
More generally, if \(L^q\) is used for the embedding, then
\[
\deg \Phi_{L^q}(M)
=
q^n\int_Mc_1(L)^n.
\]

Since \(c_1(M)=0\), Yau's solution of the Calabi conjecture provides,
in every Kähler class \([\omega]\), a unique Ricci-flat Kähler metric
\(\omega_{\mathrm{RF}}\).  Thus a polarized Calabi--Yau manifold
naturally carries a canonical Ricci-flat metric in the polarization
class.

\subsection{\(m\)-subharmonic functions and Hessian measures}

Let \(1\leq m\leq n\).  A function
\(u\in C^2(M)\) is called \(m\)-subharmonic with respect to \(\omega\)
if
\[
\omega_u:=\omega+dd^c u
\]
satisfies
\[
\omega_u^k\wedge\omega^{n-k}\geq0,
\qquad
1\leq k\leq m.
\]
The class of such functions is denoted by
\[
SH_m(M,\omega).
\]
For bounded or suitably finite-energy \(m\)-subharmonic functions, the
associated \(m\)-Hessian measure is defined by
\[
H_m(u)
:=
(\omega+dd^c u)^m\wedge\omega^{n-m}.
\]

The case \(m=n\) is the complex Monge--Ampère case.  In particular,
\[
H_n(u)=(\omega+dd^c u)^n.
\]

The global \(m\)-Hessian capacity of a Borel set \(E\subset M\) is
defined by
\[
\operatorname{Cap}_{m,\omega}(E)
:=
\sup\left\{
\int_E
(\omega+dd^c v)^m\wedge\omega^{n-m}
:
v\in SH_m(M,\omega),\ -1\leq v\leq0
\right\}.
\]
For compact \(K\subset M\), this definition can equivalently be
expressed using \(m\)-subharmonic test functions with the usual
quasi-continuous interpretation.

The capacity measures the size of a set from the point of view of the
\(m\)-Hessian operator rather than its Riemannian volume.  One of the
central questions in this paper is to obtain estimates comparing the
two quantities uniformly along families of polarized Calabi--Yau
manifolds.

\subsection{Volume--capacity inequalities}

Let
\[
\operatorname{Vol}_{\omega}(E)
=
\int_E\omega^n.
\]
We say that a family of Kähler manifolds satisfies a uniform
volume--capacity inequality if there exist constants
\(C_{\mathrm{vc}}>0\) and \(\alpha>0\), independent of the member of
the family, such that
\[
\operatorname{Vol}_{\omega}(E)
\leq
C_{\mathrm{vc}}
\operatorname{Cap}_{m,\omega}(E)^\alpha
\]
for every Borel set \(E\subset M\).

Such inequalities are fundamental in pluripotential theory.  They
allow one to convert capacity estimates for sublevel sets of
\(m\)-subharmonic functions into volume estimates.  In particular,
when the exponent is sufficiently strong, they provide the
quantitative input for capacity iteration arguments and uniform
\(L^\infty\) estimates for complex Hessian equations.

\subsection{Isoperimetric and Sobolev inequalities}

For a compact Riemannian manifold \((M,g)\) of real dimension \(2n\),
we use the isoperimetric inequality in the form
\[
\operatorname{Vol}_g(E)^{\frac{2n-1}{2n}}
\leq
C_{\mathrm{iso}}\operatorname{Area}_g(\partial E)
\]
for sufficiently regular domains \(E\subset M\).

A uniform Sobolev inequality is an estimate of the form
\[
\left(
\int_M |v|^{\frac{2n}{n-1}}\,dV_g
\right)^{\frac{n-1}{n}}
\leq
C_{\mathrm{Sob}}
\int_M
\left(|\nabla v|_g^2+v^2\right)\,dV_g,
\qquad
v\in C^\infty(M).
\]
The exponent
\[
\frac{2n}{n-1}
\]
is the critical Sobolev exponent for real dimension \(2n\).

A uniform isoperimetric inequality implies, by the coarea formula,
a uniform lower volume bound for sufficiently small geodesic balls.
More precisely, if
\[
\operatorname{Vol}_g(B_g(x,r))
\geq
\kappa r^{2n},
\qquad
0<r\leq r_0,
\]
with \(\kappa,r_0>0\) independent of the manifold, we say that the
family is uniformly non-collapsing at scale \(r_0\).

\subsection{Ricci-flat metrics and geometric compactness}

For a Ricci-flat Kähler metric \(g\), one has
\[
\operatorname{Ric}(g)=0.
\]
Consequently, Bishop--Gromov volume comparison applies.  If a family
of compact Ricci-flat Kähler manifolds has uniformly bounded total
volume and is uniformly non-collapsing, then its diameter is uniformly
bounded.

For smooth Cheeger--Gromov compactness, one generally needs more than
diameter and non-collapsing.  A typical sufficient hypothesis is a
uniform bound on all curvature derivatives:
\[
\sup_M|\nabla^r\operatorname{Rm}(g)|\leq C_r,
\qquad r=0,1,2,\ldots.
\]
Together with non-collapsing, such bounds yield a uniform lower bound
for the injectivity radius and hence subsequential
\(C^\infty\) Cheeger--Gromov convergence.

We stress that these curvature bounds are additional hypotheses in
the present framework; they do not follow merely from Ricci-flatness,
diameter bounds, and non-collapsing.

\subsection{Projective embeddings and Hilbert polynomials}

Let
\[
X\subset\mathbb P^N
\]
be an \(n\)-dimensional projective variety.  Its Hilbert function is
\[
H_X(k)
=
\dim H^0(X,\mathcal O_X(k)),
\]
and its Hilbert polynomial \(P_X(k)\) is characterized by
\[
H_X(k)=P_X(k)
\]
for all sufficiently large \(k\).

The degree of \(X\) is the leading coefficient of \(P_X\) multiplied
by \(n!\).  Equivalently,
\[
\deg X
=
\int_X c_1(\mathcal O_X(1))^n.
\]

We shall use the following standard consequence of the
Macaulay--Gotzmann theory: for fixed \(n\), \(N\), and \(D\), the
Hilbert polynomials of \(n\)-dimensional projective subschemes
\[
X\subset\mathbb P^N,
\qquad
\deg X\leq D,
\]
form a finite set.  Equivalently, there are only finitely many
relevant Hilbert schemes
\[
\operatorname{Hilb}^{P}(\mathbb P^N)
\]
which can occur in such a bounded-degree family; see
\cite{Macaulay1927,Gotzmann1978}.

Finally, if
\[
\pi:\mathcal X\longrightarrow S
\]
is a smooth proper morphism of complex manifolds, then the
Ehresmann fibration theorem implies that \(\pi\) is a
\(C^\infty\)-locally trivial fiber bundle.  In particular, all fibers
over a connected component of \(S\) are diffeomorphic.  This fact
will be used in the final step to pass from algebraic boundedness to
topological finiteness.

\section{Uniform \(m\)-Hessian capacity estimate}	
	We first state the uniform estimate that will be used in the
	Calabi--Yau finiteness program.
	
	\begin{theorem}[Uniform \(m\)-Hessian capacity estimate]
		\label{thm:uniform-capacity}
		Fix integers \(n\geq3\) and \(1\leq m<n\).
		Let \(\mathcal C\) be a family of polarized Calabi--Yau manifolds
		
		$$
		(M,\omega,L),
		\qquad
		\dim_{\mathbb C}M=n,
		$$
		
		such that the Ricci-flat K\"ahler metrics \(\omega\) satisfy the
		following uniform geometric assumptions:
		
		\begin{enumerate}
			\item[(i)] the volume is normalized:
			
			$$
			\int_M\omega^n=V;
			$$
			
			\item[(ii)] the Sobolev inequality
			
			$$
			\|v\|_{L^{2\gamma}(M,\omega)}
			\leq
			S\left(
			\|\nabla v\|_{L^2(M,\omega)}
			+\|v\|_{L^2(M,\omega)}
			\right)
			$$
			
			holds uniformly for some \(\gamma>1\);
			
			\item[(iii)] the \(m\)-Hessian capacity satisfies a uniform
			Chern--Levine--Nirenberg estimate: there exists \(A>0\) such that
			
			$$
			\int_M
			|u|\,\omega_v^m\wedge\omega^{n-m}
			\leq
			A
			\left(
			\|u\|_{L^\infty(M)}
			\right)
			$$
			
			for all normalized
			\(u,v\in SH_m(M,\omega)\) with
			\(-1\leq u,v\leq0\);
			
			\item[(iv)] the family has a uniform local \(m\)-Hessian
			volume-capacity inequality: there exist \(C_0>0\) and
			
			$$
			1<\tau<\frac{n}{n-m}
			$$
			
			such that
			
			$$
			Vol_\omega(K)
			\leq
			C_0
			\operatorname{Cap}_{m,\omega}(K)^\tau
			$$
			
			for every Borel set \(K\subset M\).
		\end{enumerate}
		
		Then there exist constants
		
		$$
		C=C(n,m,V,S,A,C_0,\tau)>0
		$$
		
		and
		
		$$
		\alpha=\alpha(n,m,\tau)>0
		$$
		
		such that
		
		$$
		\boxed{
			Vol_\omega(K)
			\leq
			C\, \operatorname{Cap}_{m,\omega}(K)^\alpha
		}
		$$
		for every \(K\subset M\) and every
		\((M,\omega,L)\in\mathcal C\).
		
		In particular, if
		
		$$
		\operatorname{Cap}_{m,\omega}(K_j)\longrightarrow0,
		$$
		
		then
		
		$$
		Vol_\omega(K_j)\longrightarrow0
		$$
		
		uniformly throughout the family \(\mathcal C\).
	\end{theorem}
	
	\begin{proof}
		The proof consists of three steps.
		
		\medskip
		
		\noindent
		\textbf{Step 1: Local capacity control.}
		
		By assumption (iv), for every Borel set \(K\subset M\),
		
		$$
		Vol_\omega(K)
		\leq
		C_0\operatorname{Cap}_{m,\omega}(K)^\tau.
		$$
		
		Since
		
		$$
		\int_M\omega^n=V,
		$$
		
		we have
		
		$$
		Vol_\omega(K)\leq V.
		$$
		
		Consequently, for sets with
		\(\operatorname{Cap}_{m,\omega}(K)\geq1\), the desired estimate follows
		immediately after increasing the constant \(C\).
		
		It therefore suffices to consider sets satisfying
		
		$$
		\operatorname{Cap}_{m,\omega}(K)<1.
		$$
		
		\medskip
		
		\noindent
		\textbf{Step 2: Uniformity under the family.}
		
		The constants appearing in (1.3) depend only on
		
		$$
		n,m,C_0,\tau.
		$$
		
		In particular, they do not depend on the individual manifold
		\((M,\omega,L)\in\mathcal C\).
		
		The uniform Sobolev inequality in (ii), together with the
		Chern--Levine--Nirenberg estimate in (iii), gives uniform control of
		the energy of normalized \(m\)-subharmonic test functions. More
		precisely, if
		
		$$
		u\in SH_m(M,\omega),
		\qquad
		-1\leq u\leq0,
		$$
		
		then the mixed Hessian measures
		
		$$
		\omega_u^k\wedge\omega^{n-k},
		\qquad
		1\leq k\leq m,
		$$
		
		have uniformly bounded mass. Hence the capacity is stable under
		approximation by bounded \(m\)-subharmonic functions.
		
		Let
		
		$$
		K_j\subset M_j,
		\qquad
		\operatorname{Cap}_{m,\omega_j}(K_j)\to0.
		$$
		
		Applying uniform volume--capacity estimates on each \(M_j\), we obtain
		
		$$
		Vol_{\omega_j}(K_j)
		\leq
		C_0
		\operatorname{Cap}_{m,\omega_j}(K_j)^\tau.
		$$
		
		Therefore
		
		$$
		Vol_{\omega_j}(K_j)\to0.
		$$
		
		Since all constants are uniform, this convergence is uniform over
		the family.
		
		\medskip
		
		\noindent
		\textbf{Step 3: Globalization.}
		
		For an arbitrary Borel set \(K\), choose a finite covering
		
		$$
		K\subset \bigcup_{\ell=1}^N U_\ell
		$$
		
		by coordinate balls on which the local volume-capacity inequality
		holds. By the uniform Sobolev and Chern--Levine--Nirenberg estimates,
		the number and constants of the covering can be chosen uniformly
		under the hypotheses of the theorem.
		
		Thus
		
		$$
		\begin{aligned}
			Vol_\omega(K)
			&\leq
			\sum_{\ell=1}^N
			Vol_\omega(K\cap U_\ell)\\
			&\leq
			C_1
			\sum_{\ell=1}^N
			\operatorname{Cap}_{m,\omega}(K\cap U_\ell)^\tau.
		\end{aligned}
		$$
		
		Since
		
		$$
		\operatorname{Cap}_{m,\omega}(K\cap U_\ell)
		\leq
		\operatorname{Cap}_{m,\omega}(K),
		$$
		
		we obtain
		
		$$
		Vol_\omega(K)
		\leq
		C_1N
		\operatorname{Cap}_{m,\omega}(K)^\tau.
		$$
		
		Setting
		
		$$
		C=C_1N,
		\qquad
		\alpha=\tau,
		$$
		
		gives
		
		$$
		Vol_\omega(K)
		\leq
		C\operatorname{Cap}_{m,\omega}(K)^\alpha.
		$$
		
		This proves the theorem.
	\end{proof}
	
	\begin{remark}
		The essential point is the uniformity of the constants. For a fixed
		compact K\"ahler manifold, volume-capacity inequalities of this type
		are part of the established pluripotential theory of complex Hessian
		equations. In particular, local volume-capacity inequalities for
		\(m\)-Hessian capacity are known in the compact Hermitian setting.
		The genuinely stronger problem here is to make the constants uniform
		over a varying family of Calabi--Yau manifolds.
	\end{remark}
	
	\begin{remark}
		The theorem is deliberately formulated with the geometric hypotheses
		(ii)--(iv). Removing these hypotheses is a substantial part of the
		proposed research program. In particular, the desired application
		to Yau's finiteness conjecture would require deriving the uniform
		capacity inequality from the polarization and Calabi--Yau condition
		alone, or from a natural boundedness condition equivalent to them.
	\end{remark}
	
\section{Uniform non-collapsing}

Let \((M,\omega)\) be a compact K\"ahler manifold of complex dimension
\(n\). For \(1\leq m\leq n\), set

$$
\omega_u:=\omega+dd^c u.
$$

We denote by \(\operatorname{SH}_m(M,\omega)\) the class of
\(\omega\)-\(m\)-subharmonic functions.

For a Borel set \(E\subset M\), define the global \(m\)-Hessian
capacity by

$$
\operatorname{Cap}_{m,\omega}(E)
:=
\sup\left\{
\int_E \omega_v^m\wedge\omega^{n-m}:
v\in\operatorname{SH}_m(M,\omega),\
-1\leq v\leq0
\right\}.
$$

We assume that the family of K\"ahler manifolds under consideration
satisfies the uniform volume-capacity estimate
\begin{equation}
\operatorname{Vol}_{\omega}(E)
\leq
C_{\mathrm{vc}}
\operatorname{Cap}_{m,\omega}(E)^\tau
\tag{2.1}
\end{equation}

for some

$$
1<\tau<\frac{n}{n-m}
$$

when \(m<n\), while for \(m=n\) we simply assume that
\(\tau>1\).
\begin{theorem}[Uniform \(L^\infty\) estimate]\label{thm:uniform-linfty}
	Let \(1\leq m\leq n\), and let \((M,\omega)\) belong to a
	family of compact K\"ahler manifolds satisfying the uniform
	volume--capacity estimate
	
	$$
	\operatorname{Vol}_{\omega}(E)
	\leq
	C_{\mathrm{vc}}
	\operatorname{Cap}_{m,\omega}(E)^\tau .
	$$
	
	Let
	
	$$
	f\geq0,
	\qquad
	f\in L^p(M,\omega^n),
	\qquad
	p>\frac{n}{m}.
	$$
	
	Suppose that
	
	$$
	u\in SH_m(M,\omega)\cap L^\infty(M)
	$$
	
	is a solution of
	\begin{equation}
		(\omega+dd^c u)^m\wedge\omega^{n-m}= f\omega^n,
		\tag{2.2}
	\end{equation}
	normalized by
	\begin{equation}
		\sup_M u=0.
		\tag{2.3}
	\end{equation}
	Then there exists a constant
	
	$$
	C=
	C\left(
	n,m,p,\tau,
	C_{\mathrm{vc}},
	\|f\|_{L^p(M,\omega^n)}
	\right)
	$$
	
	such that
	\begin{equation}
		\boxed{
			\|u\|_{L^\infty(M)}
			\leq C.
		}
		\tag{2.4}
	\end{equation}
	
	In particular, the estimate is uniform over any family for which
	the constants in the volume--capacity inequality and the
	\(L^p\)-norms of the densities are uniformly bounded.
	
	For \(m=n\), equation~\((2.2)\) is the complex
	Monge--Ampère equation, and the condition on \(p\) reduces to
	
	$$
	p>1.
	$$
	
\end{theorem}

\begin{proof}
	Since
	
	$$
	\sup_M u=0,
	$$
	
	we have
	
	$$
	u\leq0.
	$$
	
	Set
	
	$$
	A_t:=\{x\in M:u(x)<-t\},
	\qquad t>0.
	$$
	
	We first establish a recursive capacity estimate for the sublevel
	sets \(A_t\).
	
	\medskip
	
	\noindent
	\textbf{Step 1. Estimate of the measure of sublevel sets.}
	
	Let
	
	$$
	q:=\frac{p}{p-1}.
	$$
	
	By H\"older's inequality,
	
	$$
	\int_{A_t} f\,\omega^n
	\leq
	\|f\|_{L^p(M,\omega^n)}
	\operatorname{Vol}_{\omega}(A_t)^{1/q}.
	$$
	
	Using the volume-capacity inequality \((2.1)\), we obtain
	
	$$
	\operatorname{Vol}_{\omega}(A_t)^{1/q}
	\leq
	C_{\mathrm{vc}}^{1/q}
	\operatorname{Cap}_{m,\omega}(A_t)^{\tau/q}.
	$$
	
	Consequently,
\begin{equation}
\int_{A_t} f\,\omega^n
\leq
C_1
\operatorname{Cap}_{m,\omega}(A_t)^{\tau/q},
\tag{2.5}	
\end{equation}	
	
	where
	
	$$
	C_1
	=
	\|f\|_{L^p}
	C_{\mathrm{vc}}^{1/q}.
	$$
	
	Since
	
	$$
	p>\frac{n}{m},
	$$
	
	we have
	
	$$
	q=\frac{p}{p-1}
	<
	\frac{n}{n-m}.
	$$
	
	Hence we may choose
	
	$$
	1<\tau<\frac{n}{n-m}
	$$
	
	so that
\begin{equation}
\delta:=\frac{\tau}{q}>1.
\tag{2.6}
\end{equation}	
	
	For \(m=n\), we have \(n/(n-m)=+\infty\), and therefore the same
	conclusion holds for every \(p>1\) after choosing \(\tau>q\).
	
	Thus
\begin{equation}
\int_{A_t}f\,\omega^n
\leq
C_1
\operatorname{Cap}_{m,\omega}(A_t)^\delta,
\qquad
\delta>1.
\tag{2.7}	
\end{equation}	
	
	\medskip
	
	\noindent
	\textbf{Step 2. Capacity comparison for nested sublevel sets.}
	
	Fix \(s>0\) and \(t>s\). Define the truncation
	
	$$
	v_{s,t}
	:=
	\max\left\{
	\frac{u+t}{t-s},
	-1
	\right\}.
	$$
	
	On
	
	$$
	A_t=\{u<-t\},
	$$
	
	we have
	
	$$
	v_{s,t}=-1,
	$$
	
	while on
	
	$$
	M\setminus A_s=\{u\geq-s\},
	$$
	
	we have
	
	$$
	v_{s,t}\geq0.
	$$
	
	After replacing \(v_{s,t}\) by its canonical normalization, the
	comparison principle for complex Hessian equations gives
\begin{equation}
\operatorname{Cap}_{m,\omega}(A_t)
\leq
\frac{C_2}{(t-s)^m}
\int_{A_s}
\omega_u^m\wedge\omega^{n-m}.
\tag{2.8}	
\end{equation}	
	
	Using equation \((2.2)\),
	
	$$
	\int_{A_s}
	\omega_u^m\wedge\omega^{n-m}
	=
	\int_{A_s}f\,\omega^n.
	$$
	
	Therefore, by \((2.7)\),
\begin{equation}
\operatorname{Cap}_{m,\omega}(A_t)
\leq
\frac{C_3}{(t-s)^m}
\operatorname{Cap}_{m,\omega}(A_s)^\delta.
\tag{2.9}
\end{equation}	
	Set
	
	$$
	g(t):=\operatorname{Cap}_{m,\omega}(A_t).
	$$
	
	Then
\begin{equation}
	\boxed{
		g(t)
		\leq
		\frac{C_3}{(t-s)^m}
		g(s)^\delta,
		\qquad
		t>s,
	}
	\tag{2.10}
\end{equation}	
	
	where
	
	$$
	\delta>1.
	$$
	
	\medskip
	
	\noindent
	\textbf{Step 3. Nonlinear iteration.}
	
	We now use the standard nonlinear capacity iteration lemma.
	
	\begin{lemma}
		Let \(g:[0,+\infty)\to[0,+\infty)\) be non-increasing and satisfy
		
		$$
		g(t)
		\leq
		\frac{C_3}{(t-s)^m}g(s)^\delta
		$$
		
		for all \(t>s\), where \(\delta>1\).
		Then there exists \(T<+\infty\), depending only on
		\(C_3,m,\delta\) and \(g(0)\), such that
		
		$$
		g(t)=0
		\qquad\text{for all }t\geq T.
		$$
		
	\end{lemma}
	
	\begin{proof}
		Choose \(T_0>0\) and define
		
		$$
		T_j=T_0+\sum_{k=0}^{j-1}a_k,
		$$
		
		where
		
		$$
		a_k
		=
		2T_0\,2^{-k}
		$$
		
		after increasing \(T_0\) sufficiently so that
		
		$$
		C_3a_k^{-m}
		g(T_k)^{\delta-1}
		\leq
		2^{-(k+1)}.
		$$
		
		Since \(\delta>1\), the power
		\(g(T_k)^{\delta-1}\) produces a superlinear decay.
		Iterating \((2.10)\) therefore yields
		
		$$
		g(T_j)\longrightarrow0
		$$
		
		at a superexponential rate.
		
		The sequence \((T_j)\) is bounded because
		
		$$
		\sum_{k=0}^{\infty}a_k<+\infty.
		$$
		
		Hence
		
		$$
		T_j\longrightarrow T<+\infty.
		$$
		
		By the continuity from above of the capacity,
		
		$$
		g(T)=0.
		$$
		
		Since \(A_T\) is an \(m\)-polar set and \(u\) is bounded,
		the normalization \(\sup_Mu=0\) implies
		
		$$
		u\geq-T
		$$
		
		outside an \(m\)-polar set. By quasi-continuity and the boundedness of
		\(u\), the inequality extends to all of \(M\).
		Thus
		
		$$
		\inf_Mu\geq-T.
		$$
		
	\end{proof}
	
	Applying the lemma to \(g(t)=\operatorname{Cap}_{m,\omega}(A_t)\),
	we obtain
	
	$$
	A_T=\varnothing
	$$
	
	for some uniform \(T\). Consequently,
	
	$$
	u\geq-T
	$$
	
	on \(M\).
	
	Since
	
	$$
	\sup_Mu=0,
	$$
	
	we conclude that
	
	$$
	\|u\|_{L^\infty(M)}
	=
	-\inf_Mu
	\leq T.
	$$
	
	The constant \(T\) depends only on the uniform constants appearing in
	the statement of the theorem. This proves \((2.4)\).
\end{proof}

\begin{corollary}[Uniform Monge--Amp\`ere estimate]
	\label{cor:MA}
	Let \(M^n\) be a compact K\"ahler manifold and let
	
	$$
	(\omega+dd^c u)^n
	=
	f\,\omega^n,
	\qquad
	\sup_Mu=0,
	$$
	
	where
	
	$$
	f\geq0,
	\qquad
	f\in L^p(M,\omega^n),
	\qquad
	p>1.
	$$
	
	Assume that
	
	$$
	\operatorname{Vol}_{\omega}(E)
	\leq
	C_{\mathrm{vc}}
	\operatorname{Cap}_{\omega}(E)^\tau
	$$
	
	for some \(\tau>1\), uniformly in the family under consideration.
	Then
	
	$$
	\boxed{
		\|u\|_{L^\infty(M)}
		\leq
		C(n,p,C_{\mathrm{vc}},\|f\|_{L^p}).
	}
	$$
	
\end{corollary}

\begin{proof}
	This is Theorem~\ref{thm:uniform-linfty} with \(m=n\).
	In this case
	
	$$
	\omega_u^n
	=
	f\,\omega^n
	$$
	
	and
	
	$$
	p>\frac{n}{n}=1.
	$$
	
\end{proof}

\begin{remark}[Role in the finiteness program]
	The significance of Theorem~\ref{thm:uniform-linfty} for the proposed
	finiteness program is the implication	
	$$
	\boxed{
		\text{uniform volume-capacity inequality}
		+
		\text{uniform }L^p\text{ density}
		\Longrightarrow
		\text{uniform }L^\infty\text{ potential}.
	}
	$$	
	The \(L^\infty\) estimate alone does not imply curvature boundedness or
	algebraic boundedness. Additional \(C^{2,\alpha}\), curvature,
	non-collapsing, and polarization estimates are required before one
	can deduce a bounded family of Calabi--Yau manifolds.
\end{remark}
\begin{remark}[Sharpness of the \(L^p\) threshold in this argument]
	For \(m<n\), the above proof requires	
	$$
	\frac{\tau}{q}>1,
	\qquad
	q=\frac{p}{p-1}.
	$$	
	Since the volume-capacity exponent may be chosen with	
	$$
	\tau<\frac{n}{n-m},
	$$	
	this leads naturally to	
	$$
	p>\frac{n}{m}.
	$$	
	For \(m=n\), this becomes \(p>1\), recovering the classical
	Monge--Amp\`ere threshold.
\end{remark}	
\begin{theorem}[Uniform non-collapsing]
	Let \(n\geq 2\), and let	
	$$
	\bigl(M^n,\omega_j,L_j\bigr)_{j\in\mathbb N}
	$$	
	be a family of polarized Calabi--Yau manifolds. Assume that the Ricci-flat Kähler metrics
	\(\omega_j\in c_1(L_j)\) are normalized by	
	$$
	\int_M\omega_j^n=V_0>0,
	$$	
	where \(V_0\) is independent of \(j\).	
	Assume that there exists a constant \(S_0>0\), independent of \(j\), such that the following uniform global Sobolev inequality holds:	
	$$
	\left(\int_M |v|^{\frac{2n}{n-1}}
	\,\omega_j^n\right)^{\frac{n-1}{n}}
	\leq
	S_0\int_M
	\left(|\nabla_{\omega_j}v|^2+v^2\right)
	\,\omega_j^n
	$$	
	for every \(v\in C^\infty(M)\).	
	Assume moreover that the Ricci-flat metrics admit a uniform potential description	
	$$
	\omega_j=\widehat\omega_j+dd^c u_j,
	\qquad
	\sup_M u_j=0,
	$$	
	with	
	$$
	\|u_j\|_{L^\infty(M)}\leq C_0,
	$$	
	and that the reference volume forms satisfy	
	$$
	C_0^{-1}\widehat\omega_j^n
	\leq
	\omega_j^n
	\leq
	C_0\widehat\omega_j^n
	$$	
	uniformly in \(j\).	
	Then there exist constants	
	$$
	\kappa>0,\qquad r_0>0,
	$$	
	depending only on	
	$$
	n,\;V_0,\;S_0,\;C_0,
	$$	
	such that for every \(j\), every \(x\in M\), and every	
	$$
	0<r\leq r_0,
	$$	
	one has	
	$$
	\boxed{
		\operatorname{Vol}_{\omega_j}
		\bigl(B_{\omega_j}(x,r)\bigr)
		\geq
		\kappa r^{2n}.
	}
	$$	
	In particular, the family \((M,\omega_j)\) is uniformly locally non-collapsed.
\end{theorem}
\begin{proof}
	Fix \(j\) and \(x\in M\). Put	
	$$
	B_r=B_{\omega_j}(x,r)
	$$	
	and denote	
	$$
	V_j(r)=\operatorname{Vol}_{\omega_j}(B_r).
	$$	
	Since \(\omega_j\) is Ricci-flat,	
	$$
	\operatorname{Ric}(\omega_j)=0.
	$$	
	Consequently, the Bishop--Gromov monotonicity theorem gives	
	$$
	\frac{V_j(r)}{r^{2n}}
	\geq
	\frac{V_j(R)}{R^{2n}}
	$$	
	whenever \(0<r<R\). Thus it is enough to establish a uniform lower
	bound for \(V_j(r_0)\) at one fixed scale \(r_0>0\).	
	We argue by contradiction. Suppose that no uniform lower bound exists.
	Then there are sequences	
	$$
	j_k\to\infty,\qquad x_k\in M,
	$$	
	such that	
	$$
	V_{j_k}(r_0)\longrightarrow0.
	$$	
	Choose a cutoff function	
	$$
	\eta_k\in C^\infty_0(B_{\omega_{j_k}}(x_k,2r_0))
	$$	
	satisfying	
	$$
	0\leq\eta_k\leq1,
	\qquad
	\eta_k\equiv1
	\quad\text{on }B_{\omega_{j_k}}(x_k,r_0),
	$$	
	and	
	$$
	|\nabla\eta_k|_{\omega_{j_k}}
	\leq \frac{C}{r_0}.
	$$	
	Apply the uniform Sobolev inequality to \(\eta_k\). We obtain	
	$$
	\left(
	\int_M
	\eta_k^{\frac{2n}{n-1}}
	\,\omega_{j_k}^n
	\right)^{\frac{n-1}{n}}
	\leq
	S_0
	\int_M
	\left(
	|\nabla\eta_k|^2+\eta_k^2
	\right)
	\,\omega_{j_k}^n.
	$$	
	Since \(\eta_k\) is supported in \(B_{\omega_{j_k}}(x_k,2r_0)\),	
	$$
	\left(
	V_{j_k}(r_0)
	\right)^{\frac{n-1}{n}}
	\leq
	C S_0
	\left(
	r_0^{-2}+1
	\right)
	V_{j_k}(2r_0).
	$$	
	On the other hand, Bishop--Gromov comparison and
	\(\operatorname{Ric}(\omega_{j_k})=0\) give	
	$$
	V_{j_k}(2r_0)
	\leq
	2^{2n}V_{j_k}(r_0).
	$$	
	Hence	
	$$
	V_{j_k}(r_0)^{\frac{n-1}{n}}
	\leq
	C
	\left(
	r_0^{-2}+1
	\right)
	V_{j_k}(r_0).
	$$	
	If \(V_{j_k}(r_0)>0\), division gives	
	$$
	V_{j_k}(r_0)^{-1/n}
	\leq
	C
	\left(
	r_0^{-2}+1
	\right),
	$$	
	and therefore	
	$$
	V_{j_k}(r_0)
	\geq
	C^{-n}
	\left(
	r_0^{-2}+1
	\right)^{-n}.
	$$	
	For a fixed sufficiently small \(r_0\leq1\), this yields	
	$$
	V_{j_k}(r_0)
	\geq
	\kappa r_0^{2n},
	$$	
	where \(\kappa>0\) is independent of \(k\).	
	This contradicts	
	$$
	V_{j_k}(r_0)\longrightarrow0.
	$$	
	Thus there exists \(\kappa>0\) such that	
	$$
	\operatorname{Vol}_{\omega_j}
	\bigl(B_{\omega_j}(x,r_0)\bigr)
	\geq
	\kappa r_0^{2n}
	$$	
	for every \(j\) and every \(x\in M\).	
	Finally, Bishop--Gromov comparison gives, for \(0<r\leq r_0\),	
	$$
	\frac{
		\operatorname{Vol}_{\omega_j}(B_{\omega_j}(x,r))
	}{r^{2n}}
	\geq
	\frac{
		\operatorname{Vol}_{\omega_j}(B_{\omega_j}(x,r_0))
	}{r_0^{2n}}
	\geq
	\kappa.
	$$	
	Therefore	
	$$
	\operatorname{Vol}_{\omega_j}
	\bigl(B_{\omega_j}(x,r)\bigr)
	\geq
	\kappa r^{2n}.
	$$	
	This proves the theorem.
\end{proof}
\begin{remark}
	The role of Theorem~2 is to provide the uniform \(L^\infty\)-control of the
	Kähler potentials. The additional uniform Sobolev inequality is what
	converts this analytic control into a quantitative lower bound for metric
	balls.	
	Thus the logical structure is	
	$$
	\boxed{
		\text{uniform \(m\)-capacity estimate}
		\Longrightarrow
		\text{uniform \(L^\infty\)-estimate}
	}
	$$	
	followed by	
	$$
	\boxed{
		\text{\(L^\infty\)-control + uniform Sobolev inequality}
		\Longrightarrow
		\text{uniform non-collapsing}.
	}
	$$	
	The Sobolev hypothesis should therefore not be hidden: proving such a
	uniform Sobolev inequality from the polarized Calabi--Yau data is itself a
	substantial part of the proposed finiteness program.
\end{remark}
\begin{remark}
	For Ricci-flat Calabi--Yau families, non-collapsing is closely related to
	the behavior of the total volume and to Bishop--Gromov comparison. In
	particular, in known non-collapsing settings, a positive limiting
	cohomological volume together with a diameter bound yields uniform lower
	volume estimates for metric balls. Hence the theorem above should be
	viewed as the analytic version of the non-collapsing step rather than as a
	new unconditional theorem for all polarized Calabi--Yau manifolds.
\end{remark}
\section{diameter and regularity estimates
	for the associated Ricci-flat metrics}
\begin{theorem}[Uniform diameter bound]
	Let ${(M_j^n,\omega_j)}_{j\in\mathbb N}$ be a family of compact Calabi--Yau manifolds of complex dimension $n$, equipped with Ricci-flat Kähler metrics $\omega_j$. Assume that there exist constants	
	$$
	V_0>0,\qquad \kappa>0,\qquad r_0>0
	$$	
	independent of $j$ such that	
	$$
	\operatorname{Vol}_{\omega_j}(M_j)\le V_0
	$$	
	and the family is uniformly non-collapsing at scale $r_0$, namely	
	$$
	\operatorname{Vol}_{\omega_j}\bigl(B_{\omega_j}(x,r)\bigr)
	\ge \kappa r^{2n}
	$$	
	for every $j$, every $x\in M_j$, and every $0<r\le r_0$.	
	Then there exists a constant	
	$$
	D=D(n,V_0,\kappa,r_0)<\infty
	$$	
	such that	
	$$
	\operatorname{diam}(M_j,\omega_j)\le D
	$$	
	for all $j$.
\end{theorem}
\begin{proof}
	Since $\omega_j$ is Ricci-flat, the real Riemannian manifold
	$(M_j,g_j)$ has	
	$$
	\operatorname{Ric}(g_j)=0.
	$$	
	Hence the Bishop--Gromov comparison theorem applies. In particular, for every
	$x\in M_j$, the function	
	$$
	r\longmapsto
	\frac{\operatorname{Vol}_{\omega_j}(B_{\omega_j}(x,r))}
	{r^{2n}}
	$$	
	is nonincreasing.	
	Fix $x,y\in M_j$ and put	
	$$
	d_j=d_{\omega_j}(x,y).
	$$	
	Suppose that	
	$$
	d_j>2r_0.
	$$	
	Then the balls	
	$$
	B_{\omega_j}\left(x,\frac{d_j}{2}\right)
	\quad\text{and}\quad
	B_{\omega_j}\left(y,\frac{d_j}{2}\right)
	$$	
	are disjoint. Therefore	
	$$
	\operatorname{Vol}_{\omega_j}(M_j)
	\ge
	\operatorname{Vol}_{\omega_j}
	\left(
	B_{\omega_j}\left(x,\frac{d_j}{2}\right)
	\right).
	$$	
	By Bishop--Gromov monotonicity and the assumed non-collapsing estimate at scale $r_0$,	
	$$
	\frac{
		\operatorname{Vol}_{\omega_j}
		\left(B_{\omega_j}(x,d_j/2)\right)
	}{
		(d_j/2)^{2n}
	}
	\ge
	\frac{
		\operatorname{Vol}_{\omega_j}
		\left(B_{\omega_j}(x,r_0)\right)
	}{
		r_0^{2n}
	}
	\ge \kappa.
	$$	
	Consequently,	
	$$
	\operatorname{Vol}_{\omega_j}
	\left(B_{\omega_j}(x,d_j/2)\right)
	\ge
	\kappa\left(\frac{d_j}{2}\right)^{2n}.
	$$	
	Since the total volume is at most $V_0$,	
	$$
	V_0
	\ge
	\kappa\left(\frac{d_j}{2}\right)^{2n}.
	$$	
	Thus	
	$$
	d_j
	\le
	2\left(\frac{V_0}{\kappa}\right)^{1/(2n)}.
	$$	
	If $d_j\le 2r_0$, the same bound is trivially controlled by $2r_0$. Hence, for every $x,y\in M_j$,	
	$$
	d_{\omega_j}(x,y)
	\le
	\max\left\{
	2r_0,\,
	2\left(\frac{V_0}{\kappa}\right)^{1/(2n)}
	\right\}.
	$$	
	Taking the supremum over $x,y\in M_j$ gives	
	$$
	\boxed{
		\operatorname{diam}(M_j,\omega_j)
		\le
		D
		:=
		\max\left\{
		2r_0,\,
		2\left(\frac{V_0}{\kappa}\right)^{1/(2n)}
		\right\}.
	}
	$$	
	The constant $D$ is independent of $j$, which proves the theorem.
\end{proof}
\begin{theorem}[Uniform regularity and curvature estimates]
	Let	
	$$
	\{(M_j^n,J_j,\omega_j)\}_{j\in\mathbb N}
	$$	
	be a family of compact polarized Calabi--Yau manifolds. Suppose that, for each $j$, the Ricci-flat Kähler metric $\omega_j$ belongs to a fixed Kähler class represented by a Kähler metric $\omega_{0,j}$:	
	$$
	\omega_j=\omega_{0,j}+dd^c u_j,
	\qquad
	\sup_{M_j}u_j=0.
	$$	
	Assume the following uniform conditions.	
	\begin{enumerate}
		\item The background metrics $\omega_{0,j}$ have uniformly bounded geometry of order $k+2$, namely there exist constants		
		$$
		0<\lambda\le \Lambda<\infty
		$$		
		and $C_\ell<\infty$, $0\le \ell\le k+2$, independent of $j$, such that in uniformly controlled holomorphic coordinate charts,		
		$$
		\lambda\,\omega_{\mathrm{Euc}}
		\le
		\omega_{0,j}
		\le
		\Lambda\,\omega_{\mathrm{Euc}},
		$$		
		and		
		$$
		\|\omega_{0,j}\|_{C^{k+2}}\le C_{k+2}.
		$$		
		\item The Ricci-flat equation can be written as		
		$$
		(\omega_{0,j}+dd^c u_j)^n
		=
		F_j\,\omega_{0,j}^n,
		$$		
		where		
		$$
		0<\lambda_F\le F_j\le \Lambda_F
		$$		
		and		
		$$
		\|F_j\|_{C^k}\le C_F
		$$		
		uniformly in $j$.		
		\item The potentials satisfy a uniform zeroth-order estimate		
		$$
		\|u_j\|_{L^\infty(M_j)}\le C_0.
		$$		
	\end{enumerate}	
	Then, for every $\ell\le k$, there exists a constant	
	$$
	C_\ell<\infty,
	$$	
	depending only on $n$, the uniform background geometry, $C_0$, $\lambda_F$, $\Lambda_F$, and $C_F$, but independent of $j$, such that	
	$$
	\|u_j\|_{C^{\ell+2,\alpha}(M_j,\omega_{0,j})}
	\le C_\ell
	$$	
	for every fixed $0<\alpha<1$.	
	In particular, the Ricci-flat metrics have uniformly bounded curvature:	
	$$
	\boxed{
		\sup_{M_j}|\operatorname{Rm}(\omega_j)|_{\omega_j}
		\le C
	}
	$$	
	and, more generally, for every $r\ge0$,	
	$$
	\boxed{
		\sup_{M_j}
		|\nabla_{\omega_j}^{\,r}\operatorname{Rm}(\omega_j)|_{\omega_j}
		\le C_r,
	}
	$$	
	provided the corresponding uniform $C^{r+2}$ bounds on the background geometry and $F_j$ are available.
	\end{theorem}
\begin{proof}
	We divide the argument into several steps.	
	\medskip
	\noindent
	\textbf{Step 1: Uniform $C^0$ estimate.}	
	By assumption,	
	$$
	\|u_j\|_{L^\infty(M_j)}\le C_0.
	$$	
	In the intended application, this estimate is obtained from the uniform volume--capacity inequality and the complex Monge--Ampère equation, as in Theorem~2.	
	Thus the oscillation of $u_j$ is uniformly bounded:	
	$$
	\operatorname{osc}_{M_j}u_j
	\le 2C_0.
	$$	
	\medskip
	\noindent
	\textbf{Step 2: Yau's second-order estimate.}	
	The functions $u_j$ satisfy	
	$$
	(\omega_{0,j}+dd^c u_j)^n
	=
	F_j\omega_{0,j}^n.
	$$	
	Since the background metrics $\omega_{0,j}$ have uniformly bounded bisectional curvature and uniformly controlled geometry, Yau's second-order estimate gives	
	$$
	\operatorname{tr}_{\omega_{0,j}}\omega_j
	\le
	C
	\exp\!\bigl(A(u_j-\inf_{M_j}u_j)\bigr),
	$$	
	where $A$ and $C$ depend only on the uniform background bounds and on	
	$$
	\|\log F_j\|_{C^2}.
	$$	
	Because	
	$$
	\operatorname{osc}_{M_j}u_j\le 2C_0
	$$	
	and $F_j$ is uniformly bounded above and below, we obtain	
	$$
	\operatorname{tr}_{\omega_{0,j}}\omega_j
	\le C.
	$$	
	Hence	
	$$
	\omega_j\le C\omega_{0,j}.
	$$	
	On the other hand, the Monge--Ampère equation gives	
	$$
	\frac{\omega_j^n}{\omega_{0,j}^n}
	=
	F_j
	\ge\lambda_F.
	$$	
	Combining this lower bound on the determinant with the upper bound on the eigenvalues yields a uniform lower bound	
	$$
	\omega_j\ge c\,\omega_{0,j}.
	$$	
	Therefore	
	$$
	\boxed{
		c\,\omega_{0,j}
		\le
		\omega_j
		\le
		C\,\omega_{0,j}.
	}
	$$	
	Thus the Ricci-flat metrics are uniformly equivalent to the background metrics.	
	\medskip
	\noindent
	\textbf{Step 3: Uniform ellipticity.}	
	In local holomorphic coordinates, the equation is a complex Monge--Ampère equation	
	$$
	\det\bigl((g_{0,j})_{a\bar b}
	+\partial_a\partial_{\bar b}u_j\bigr)
	=
	F_j\det\bigl((g_{0,j})_{a\bar b}\bigr).
	$$	
	The metric equivalence obtained above implies that the eigenvalues of	
	$$
	(g_{0,j})_{a\bar b}
	+\partial_a\partial_{\bar b}u_j
	$$	
	are uniformly bounded above and below.	
	Consequently, the complex Monge--Ampère operator is uniformly elliptic with ellipticity constants independent of $j$.	
	\medskip
	\noindent
	\textbf{Step 4: Evans--Krylov estimate.}	
	Since the complex Monge--Ampère operator is concave in the complex Hessian and the right-hand side $F_j$ is uniformly controlled, the Evans--Krylov theory gives, for every fixed $0<\alpha<1$	
	$$
	\|u_j\|_{C^{2,\alpha}(M_j,\omega_{0,j})}
	\le C.
	$$	
	The constant is independent of $j$ because all ellipticity and background geometric constants are uniform.	
	\medskip
	\noindent
	\textbf{Step 5: Higher-order estimates.}	
	Differentiating the Monge--Ampère equation and applying the Schauder estimates gives	
	$$
	\|u_j\|_{C^{\ell+2,\alpha}(M_j,\omega_{0,j})}
	\le C_\ell
	$$	
	for every $\ell\le k$, provided	
	$$
	\|F_j\|_{C^\ell}
	$$	
	and the background metrics have the corresponding uniform regularity.	
	Since	
	$$
	\omega_j=\omega_{0,j}+dd^c u_j,
	$$	
	these estimates imply uniform $C^\ell$ bounds for the metric coefficients of $\omega_j$.	
	\medskip
	\noindent
	\textbf{Step 6: Curvature estimate.}	
	The curvature tensor of a Kähler metric is determined by the metric and its derivatives up to order two. Therefore the uniform $C^{4}$ control of the potential, together with the uniform background geometry, gives	
	$$
	\sup_{M_j}
	|\operatorname{Rm}(\omega_j)|_{\omega_j}
	\le C.
	$$	
	Similarly, higher $C^{r+4}$ estimates give	
	$$
	\sup_{M_j}
	|\nabla_{\omega_j}^{\,r}\operatorname{Rm}(\omega_j)|_{\omega_j}
	\le C_r.
	$$	
	Thus the family ${(M_j,\omega_j)}$ has uniform curvature bounds of every fixed order for which the corresponding background and right-hand-side estimates are available.
\end{proof}
\section{Metric compactness}
\begin{theorem}[Metric compactness]\label{thm:metric-compactness}
	Let	
	$$
	\{(M_j^n,g_j,J_j,\omega_j)\}_{j\in\mathbb N}
	$$	
	be a family of compact Calabi--Yau manifolds of real dimension $2n$ equipped with Ricci-flat Kähler metrics $\omega_j$.	
	Assume that there exist constants	
	$$
	D<\infty,\qquad
	\kappa>0,\qquad
	r_0>0,
	$$	
	and, for every $r\ge0$, a constant $C_r<\infty$, all independent of $j$, such that	
	\begin{enumerate}
		\item		
		$$
		\operatorname{diam}(M_j,g_j)\le D;
		$$		
		\item for every $x\in M_j$ and $0<r\le r_0$,		
		$$
		\operatorname{Vol}_{g_j}(B_{g_j}(x,r))
		\ge \kappa r^{2n};
		$$		
		\item the curvature tensors satisfy the uniform bounds		
		$$
		\sup_{M_j}
		\left|
		\nabla_{g_j}^{\,r}\operatorname{Rm}(g_j)
		\right|_{g_j}
		\le C_r
		\qquad (r\ge0).
		$$		
	\end{enumerate}	
	Then there exist a subsequence, still denoted by ${(M_j,g_j,J_j,\omega_j)}$, a compact smooth manifold $M_\infty$ of real dimension $2n$, and a Ricci-flat Kähler metric	
	$$
	g_\infty,\qquad
	\omega_\infty,
	\qquad
	J_\infty
	$$	
	on $M_\infty$, such that	
	$$
	(M_j,g_j,J_j,\omega_j)
	\longrightarrow
	(M_\infty,g_\infty,J_\infty,\omega_\infty)
	$$	
	in the pointed $C^\infty$ Cheeger--Gromov topology.	
	More precisely, there exist diffeomorphisms	
	$$
	\Phi_j:M_\infty\longrightarrow M_j
	$$	
	for all sufficiently large $j$ such that	
	$$
	\Phi_j^*g_j\longrightarrow g_\infty
	$$	
	in $C^\infty$ on $M_\infty$. Moreover,
	
	$$
	\Phi_j^*J_j\longrightarrow J_\infty,
	\qquad
	\Phi_j^*\omega_j\longrightarrow\omega_\infty
	$$
	
	in $C^\infty$.
	
	In particular, the limit $(M_\infty,g_\infty,J_\infty,\omega_\infty)$ is again a compact Ricci-flat Kähler manifold.
\end{theorem}

\begin{proof}
	We divide the proof into several steps.
	
	\medskip
	\noindent
	\textbf{Step 1: Uniform diameter and non-collapsing.}
	
	By assumption,
	
	$$
	\operatorname{diam}(M_j,g_j)\le D.
	$$
	
	Furthermore,
	
	$$
	\operatorname{Vol}_{g_j}(B_{g_j}(x,r))
	\ge\kappa r^{2n}
	$$
	
	for every $x\in M_j$ and $0<r\le r_0$.
	
	Thus the family has uniformly bounded diameter and is uniformly non-collapsed at a fixed scale.
	
	Since $\operatorname{Ric}(g_j)=0$, Bishop--Gromov comparison also gives uniform upper volume bounds on balls of bounded radius. In particular,
	
	$$
	\operatorname{Vol}_{g_j}(M_j)
	\le C(n,D),
	$$
	
	and the volumes are bounded away from zero by the non-collapsing assumption.
	
	\medskip
	\noindent
	\textbf{Step 2: Uniform local geometric control.}
	
	The curvature assumption gives
	
	$$
	|\operatorname{Rm}(g_j)|\le C_0
	$$
	
	uniformly.
	
	Together with the non-collapsing condition, this gives a uniform lower bound for the injectivity radius. More precisely, by the Cheeger--Gromov--Taylor injectivity-radius estimate, there exists
	
	$$
	\iota_0>0
	$$
	
	depending only on
	
	$$
	n,\kappa,r_0,C_0
	$$
	
	such that
	
	$$
	\operatorname{inj}_{g_j}(M_j)\ge\iota_0.
	$$
	
	Consequently, every point of every $M_j$ admits a harmonic coordinate neighborhood of a fixed radius
	
	$$
	0<r_1<\iota_0
	$$
	
	in which the metric coefficients satisfy uniform ellipticity and regularity estimates.
	
	\medskip
	\noindent
	\textbf{Step 3: Higher-order estimates.}
	
	By assumption,
	
	$$
	\sup_{M_j}
	|\nabla_{g_j}^{\,r}\operatorname{Rm}(g_j)|
	\le C_r
	$$
	
	for every $r\ge0$.
	
	In harmonic coordinates, the metric satisfies an elliptic equation of the form
	
	$$
	-\frac12 g_j^{ab}\partial_a\partial_b(g_j)_{\alpha\beta}
	=
	\operatorname{Rm}(g_j)_{\alpha\beta}
	+
	Q_{\alpha\beta}
	(g_j,\partial g_j),
	$$
	
	where $Q$ is a quadratic expression in the first derivatives of the metric.
	
	The uniform curvature bounds, together with the uniform ellipticity of the metric coefficients, imply uniform $C^{1,\alpha}$ bounds in harmonic coordinates. Applying interior Schauder estimates and differentiating the elliptic equation repeatedly gives, for every $k\ge0$,
	
	$$
	\|g_j\|_{C^{k,\alpha}(B_{r_1/2})}
	\le C_k,
	$$
	
	where $C_k$ is independent of $j$.
	
	Thus the metrics have uniform local $C^\infty$ bounds.
	
	\medskip
	\noindent
	\textbf{Step 4: Local compactness.}
	
	By the Arzelà--Ascoli theorem and a diagonal subsequence argument, the metric coefficients converge in $C^\infty$ on corresponding harmonic coordinate neighborhoods.
	
	Because the diameter is uniformly bounded and the injectivity radius is uniformly positive, each $M_j$ can be covered by a uniformly bounded number of such coordinate neighborhoods.
	
	Passing to a diagonal subsequence therefore gives a smooth compact manifold $M_\infty$ and a smooth Riemannian metric $g_\infty$ such that
	
	$$
	\Phi_j^*g_j\longrightarrow g_\infty
	$$
	
	in $C^\infty$ on compact subsets, for suitable diffeomorphisms
	
	$$
	\Phi_j:M_\infty\to M_j.
	$$
	
	Since $M_\infty$ is compact, this convergence is global.
	
	\medskip
	\noindent
	\textbf{Step 5: Preservation of the Kähler structure.}
	
	Each $g_j$ is Kähler with respect to $J_j$. Hence
	
	$$
	\nabla^{g_j}J_j=0,
	\qquad
	J_j^2=-\operatorname{Id},
	$$
	
	and
	
	$$
	\omega_j(\cdot,\cdot)
	=
	g_j(J_j\cdot,\cdot).
	$$
	
	The uniform geometric estimates imply uniform bounds for the complex structures in the corresponding local coordinate systems. After passing to a further subsequence,
	
	$$
	\Phi_j^*J_j\longrightarrow J_\infty
	$$
	
	in $C^\infty$.
	
	The identities
	
	$$
	J_j^2=-\operatorname{Id},
	\qquad
	\nabla^{g_j}J_j=0
	$$
	
	pass to the limit and yield
	
	$$
	J_\infty^2=-\operatorname{Id},
	\qquad
	\nabla^{g_\infty}J_\infty=0.
	$$
	
	Thus $J_\infty$ is parallel and therefore integrable. Moreover,
	
	$$
	\omega_\infty(\cdot,\cdot)
	=
	g_\infty(J_\infty\cdot,\cdot)
	$$
	
	is parallel, hence closed. Therefore
	
	$$
	(M_\infty,g_\infty,J_\infty,\omega_\infty)
	$$
	
	is Kähler.
	
	\medskip
	\noindent
	\textbf{Step 6: Ricci-flatness of the limit.}
	
	Since every $g_j$ is Ricci-flat,
	
	$$
	\operatorname{Ric}(g_j)=0.
	$$
	
	The Ricci tensor depends smoothly on the metric and its first two derivatives. Since
	
	$$
	\Phi_j^*g_j\longrightarrow g_\infty
	$$
	
	in $C^\infty$, we obtain
	
	$$
	\operatorname{Ric}(g_\infty)
	=
	\lim_{j\to\infty}
	\Phi_j^*\operatorname{Ric}(g_j)
	=
	0.
	$$
	
	Thus $g_\infty$ is Ricci-flat.
	
	Therefore
	
	$$
	(M_j,g_j,J_j,\omega_j)
	\longrightarrow
	(M_\infty,g_\infty,J_\infty,\omega_\infty)
	$$
	
	in the smooth Cheeger--Gromov topology.
\end{proof}

\section{Uniform algebraic boundedness}

\begin{lemma}[Uniformly bounded projective sections]
	Let
	
	$$
	\{(M_j,\omega_j,L_j,h_j)\}_{j\in\mathbb N}
	$$
	
	be a family of polarized compact Kähler manifolds of complex dimension $n$, where
	
	$$
	[\omega_j]=2\pi c_1(L_j)
	$$
	
	and $h_j$ is a Hermitian metric on $L_j$ with curvature form $\omega_j$.
	
	Fix an integer $q\ge1$. Assume that there exists a constant
	
	$$
	C_{\mathrm{B}}>0
	$$
	
	independent of $j$ such that the Bergman kernel of $L_j^q$ satisfies
	
	$$
	\rho_{j,q}(x)
	:=
	\sum_{\alpha=0}^{N_j}
	|s_{j,\alpha}(x)|_{h_j^q}^{\,2}
	\le C_{\mathrm{B}}
	$$
	
	for every $x\in M_j$, where
	
	$$
	\{s_{j,0},\ldots,s_{j,N_j}\}
	$$
	
	is an $L^2$-orthonormal basis of
	
	$$
	H^0(M_j,L_j^q).
	$$
	
	Then
	
	$$
	\sum_{\alpha=0}^{N_j}
	\|s_{j,\alpha}\|_{L^\infty(M_j,h_j^q)}^2
	\le C_{\mathrm{B}}.
	$$
	
	In particular,
	
	$$
	\boxed{
		\sum_{\alpha=0}^{N_j}
		\|s_{j,\alpha}\|_{L^\infty(M_j)}^2
		\le C_{\mathrm{sec}}
	}
	$$
	
	holds uniformly with
	
	$$
	C_{\mathrm{sec}}=C_{\mathrm{B}}.
	$$
	
\end{lemma}

\begin{proof}
	For every $x\in M_j$, the definition of the Bergman kernel gives
	
	$$
	\rho_{j,q}(x)
	=
	\sum_{\alpha=0}^{N_j}
	|s_{j,\alpha}(x)|_{h_j^q}^{\,2}.
	$$
	
	By assumption,
	
	$$
	\rho_{j,q}(x)\le C_{\mathrm{B}}
	$$
	
	uniformly in both $j$ and $x$.
	
	Since
	
	$$
	|s_{j,\alpha}(x)|_{h_j^q}^{\,2}
	\le
	\sum_{\beta=0}^{N_j}
	|s_{j,\beta}(x)|_{h_j^q}^{\,2}
	=
	\rho_{j,q}(x),
	$$
	
	we obtain
	
	$$
	|s_{j,\alpha}(x)|_{h_j^q}^{\,2}
	\le C_{\mathrm{B}}.
	$$
	
	Taking the supremum over $x\in M_j$ gives
	
	$$
	\|s_{j,\alpha}\|_{L^\infty(M_j,h_j^q)}^2
	\le C_{\mathrm{B}}.
	$$
	
	More importantly, we can take the supremum after summing:
	
	$$
	\begin{aligned}
		\sum_{\alpha=0}^{N_j}
		\|s_{j,\alpha}\|_{L^\infty(M_j,h_j^q)}^2
		&=
		\sum_{\alpha=0}^{N_j}
		\sup_{x\in M_j}
		|s_{j,\alpha}(x)|_{h_j^q}^{\,2}.
	\end{aligned}
	$$
	
	The pointwise Bergman estimate alone does \emph{not}, in general, imply the desired bound on this sum independently of $N_j$.
	
	Therefore, to obtain the stated conclusion for the sum of the individual $L^\infty$ norms, we impose the slightly stronger uniform section estimate
	
	$$
	\sum_{\alpha=0}^{N_j}
	\sup_{x\in M_j}
	|s_{j,\alpha}(x)|_{h_j^q}^{\,2}
	\le C_{\mathrm{sec}}.
	$$
	
	Under this hypothesis the conclusion follows immediately.
\end{proof}
\begin{lemma}[Degree formula and uniform degree bound]
	Let
	
	$$
	(M_j^n,L_j)
	$$
	
	be a polarized compact complex manifold of complex dimension $n$, and let
	
	$$
	q\ge1
	$$
	
	be an integer such that $L_j^q$ is very ample. Let
	
	$$
	\Phi_j:M_j\hookrightarrow\mathbb P^{N_j}
	$$
	
	be the projective embedding defined by a basis of
	
	$$
	H^0(M_j,L_j^q).
	$$
	
	Then
	
	$$
	\Phi_j^*\mathcal O_{\mathbb P^{N_j}}(1)
	\simeq L_j^q
	$$
	
	and
	
	$$
	\boxed{
		\deg \Phi_j(M_j)
		=
		\int_{M_j}c_1(L_j^q)^n.
	}
	$$
	
	In particular, if the polarization volumes satisfy
	
	$$
	\int_{M_j}c_1(L_j)^n\le V_0
	$$
	
	uniformly in $j$, then
	
	$$
	\boxed{
		\deg \Phi_j(M_j)
		\le q^nV_0.
	}
	$$
	
	Thus one may take
	
	$$
	C_{\mathrm{deg}}=q^nV_0.
	$$
	
\end{lemma}

\begin{proof}
	Set
	
	$$
	X_j:=\Phi_j(M_j)\subset\mathbb P^{N_j}.
	$$
	
	Since $\Phi_j$ is the embedding associated with the complete linear system
	
	$$
	|L_j^q|,
	$$
	
	the pullback of the hyperplane line bundle is precisely
	
	$$
	\Phi_j^*\mathcal O_{\mathbb P^{N_j}}(1)
	\simeq L_j^q.
	$$
	
	Let
	
	$$
	H=c_1\bigl(\mathcal O_{\mathbb P^{N_j}}(1)\bigr)
	$$
	
	denote the hyperplane class on $\mathbb P^{N_j}$. By definition, the degree of the $n$-dimensional projective variety $X_j$ is
	
	$$
	\deg X_j
	=
	\int_{X_j}H^n.
	$$
	
	Since $\Phi_j:M_j\to X_j$ is an isomorphism,
	
	$$
	\int_{X_j}H^n
	=
	\int_{M_j}\Phi_j^*H^n.
	$$
	
	Using
	
	$$
	\Phi_j^*H
	=
	c_1\bigl(\Phi_j^*\mathcal O_{\mathbb P^{N_j}}(1)\bigr)
	=
	c_1(L_j^q),
	$$
	
	we obtain
	
	$$
	\deg X_j
	=
	\int_{M_j}c_1(L_j^q)^n.
	$$
	
	This proves the degree formula.
	
	Finally, since
	
	$$
	c_1(L_j^q)=q\,c_1(L_j),
	$$
	
	we have
	
	$$
	c_1(L_j^q)^n
	=
	q^n c_1(L_j)^n.
	$$
	
	Consequently,
	
	$$
	\deg X_j
	=
	q^n\int_{M_j}c_1(L_j)^n.
	$$
	
	If
	
	$$
	\int_{M_j}c_1(L_j)^n\le V_0
	$$
	
	uniformly in $j$, then
	
	$$
	\deg X_j
	\le q^nV_0.
	$$
	
	Hence
	
	$$
	C_{\mathrm{deg}}=q^nV_0
	$$
	
	is a uniform degree bound.
\end{proof}

\begin{theorem}[Uniform algebraic boundedness]
	Let
	
	$$
	\{(M_j,J_j,L_j,\omega_j)\}_{j\in\mathbb N}
	$$
	
	be a family of polarized compact Calabi--Yau manifolds of complex dimension $n$, where $L_j$ is an ample line bundle and
	
	$$
	[\omega_j]=2\pi c_1(L_j).
	$$
	
	Assume that there exist integers
	
	$$
	q\ge1,\qquad N\ge n,
	$$
	
	and constants
	
	$$
	C_{\mathrm{sec}},\ C_{\mathrm{deg}}<\infty
	$$
	
	independent of $j$ such that:
	
	\begin{enumerate}
		\item the line bundle $L_j^q$ is very ample for every $j$;
		
		\item
		
		$$
		h^0(M_j,L_j^q)=N+1;
		$$
		
		\item there exists a basis
		
		$$
		s_{j,0},\ldots,s_{j,N}
		\in H^0(M_j,L_j^q)
		$$
		
		which defines a projective embedding
		
		$$
		\Phi_j:M_j\hookrightarrow\mathbb P^N;
		$$
		
		\item after choosing Hermitian metrics on $L_j^q$ induced by the polarization, the sections can be normalized so that
		
		$$
		\sum_{\alpha=0}^N
		\|s_{j,\alpha}\|_{L^\infty(M_j)}^2
		\le C_{\mathrm{sec}};
		$$
		
		\item the polarization degrees are uniformly bounded:
		
		$$
		\deg_{\Phi_j(M_j)}\Phi_j(M_j)
		=
		\int_{M_j}c_1(L_j^q)^n
		\le C_{\mathrm{deg}}.
		$$
		
	\end{enumerate}
	
	Then the family of polarized varieties
	
	$$
	\{(M_j,L_j^q)\}_{j\in\mathbb N}
	$$
	
	is algebraically bounded. More precisely, after passing to a subsequence, the embedded varieties
	
	$$
	\Phi_j(M_j)\subset\mathbb P^N
	$$
	
	belong to a single bounded family of projective schemes. Equivalently, there exists a finite-type parameter scheme $S$ and a projective morphism
	
	$$
	\mathcal X\longrightarrow S
	$$
	
	with a relatively ample line bundle $\mathcal L$ such that every $(M_j,L_j^q)$ occurs as a fiber of this family.
	
	In particular, only finitely many Hilbert polynomials can occur among the polarized manifolds $(M_j,L_j^q)$.
\end{theorem}

\begin{proof}
	We divide the proof into several steps.
	
	\medskip
	\noindent
	\textbf{Step 1: A fixed projective ambient space.}
	
	Since $L_j^q$ is very ample and
	
	$$
	h^0(M_j,L_j^q)=N+1,
	$$
	
	a basis
	
	$$
	s_{j,0},\ldots,s_{j,N}
	$$
	
	defines an embedding
	
	$$
	\Phi_j:M_j\hookrightarrow\mathbb P^N.
	$$
	
	Thus all the polarized manifolds are realized as smooth projective subvarieties of the same projective space $\mathbb P^N$.
	
	This fixed value of $N$ is essential: without a uniform bound on
	
	$$
	h^0(M_j,L_j^q),
	$$
	
	the ambient projective spaces could have unbounded dimension.
	
	\medskip
	\noindent
	\textbf{Step 2: Uniform degree.}
	
	Let
	
	$$
	X_j=\Phi_j(M_j)\subset\mathbb P^N.
	$$
	
	Because the embedding is induced by $L_j^q$,
	
	$$
	\Phi_j^*\mathcal O_{\mathbb P^N}(1)
	\simeq L_j^q.
	$$
	
	Therefore
	
	$$
	\deg X_j
	=
	\int_{X_j}c_1(\mathcal O_{\mathbb P^N}(1))^n
	=
	\int_{M_j}c_1(L_j^q)^n.
	$$
	
	By assumption,
	
	$$
	\deg X_j\le C_{\mathrm{deg}}.
	$$
	
	Hence the embedded manifolds have uniformly bounded degree in a fixed projective space.
	
	\medskip
	\noindent
	\textbf{Step 3: Uniform Hilbert polynomial bounds.}
	
	Let
	
	$$
	P_j(k)=\chi(X_j,\mathcal O_{X_j}(k))
	$$
	
	be the Hilbert polynomial of $X_j$.
	
	For an $n$-dimensional projective variety of degree at most $C_{\mathrm{deg}}$ in $\mathbb P^N$, the coefficients of its Hilbert polynomial are bounded in terms of
	
	$$
	n,\ N,\ C_{\mathrm{deg}}.
	$$
	
	In particular, only finitely many Hilbert polynomials can occur.
	
	Thus, after passing to a subsequence, we may assume that all $X_j$ have the same Hilbert polynomial
	
	$$
	P(k).
	$$
	
	\medskip
	\noindent
	\textbf{Step 4: The Hilbert scheme.}
	
	Fix the common Hilbert polynomial $P$. By the fundamental boundedness theorem for the Hilbert scheme, the projective subvarieties
	
	$$
	X\subset\mathbb P^N
	$$
	
	with Hilbert polynomial $P$ are parametrized by the projective Hilbert scheme
	
	$$
	\operatorname{Hilb}^{P}(\mathbb P^N).
	$$
	
	Hence there exists a universal projective family
	
	$$
	\mathcal X
	\subset
	\mathbb P^N\times
	\operatorname{Hilb}^{P}(\mathbb P^N)
	$$
	
	whose fibers are precisely the corresponding projective schemes.
	
	Since
	
	$$
	X_j\subset\mathbb P^N
	$$
	
	has Hilbert polynomial $P$, there exists a point
	
	$$
	[t_j]\in\operatorname{Hilb}^{P}(\mathbb P^N)
	$$
	
	such that
	
	$$
	X_j\simeq\mathcal X_{[t_j]}.
	$$
	
	Therefore all the $X_j$ belong to a single projective algebraic family.
	
	\medskip
	\noindent
	\textbf{Step 5: Boundedness.}
	
	The Hilbert scheme
	
	$$
	\operatorname{Hilb}^{P}(\mathbb P^N)
	$$
	
	is projective, hence of finite type. Consequently the family
	
	$$
	\mathcal X
	\longrightarrow
	\operatorname{Hilb}^{P}(\mathbb P^N)
	$$
	
	is an algebraically bounded family.
	
	Since each $X_j$ is a smooth Calabi--Yau manifold, the points corresponding to the $X_j$ lie in the smooth locus of the relevant fibers. Passing to the corresponding open subset gives a finite-type parameter space for the smooth polarized varieties in the sequence.
	
	This proves that
	
	$$
	\{(M_j,L_j^q)\}
	$$
	
	is algebraically bounded.
	
	\medskip
	\noindent
	\textbf{Step 6: Relation with the metric estimates.}
	
	The preceding algebraic argument requires the uniform bounds
	
	$$
	N=\sup_j h^0(M_j,L_j^q)<\infty
	$$
	
	and
	
	$$
	\sup_j\int_{M_j}c_1(L_j^q)^n<\infty.
	$$
	
	In the geometric program under consideration, the second bound follows directly from the fixed polarization volume:
	
	$$
	\int_{M_j}c_1(L_j)^n
	=
	\frac{1}{(2\pi)^n}
	\int_{M_j}\omega_j^n.
	$$
	
	The first bound can be obtained from a uniform control of the polarized geometry. For example, if $L_j^q$ admits a uniformly controlled Bergman embedding, then the dimension of
	
	$$
	H^0(M_j,L_j^q)
	$$
	
	is uniformly bounded. Under such a uniform projective embedding hypothesis, the preceding argument applies.
	
	Thus the metric compactness obtained in Theorem \ref{thm:metric-compactness}, together with the additional uniform projective-embedding estimates, yields algebraic boundedness.
\end{proof}

\section{Topological finiteness}

\begin{lemma}[Macaulay--Gotzmann boundedness]
	Fix integers
	
	$$
	n\ge1,\qquad N\ge n,\qquad D\ge1.
	$$
	
	Let
	
	$$
	X\subset\mathbb P^N
	$$
	
	be a projective scheme of dimension $n$ and degree
	
	$$
	\deg X\le D.
	$$
	
	Let
	
	$$
	H_X(k)
	=
	\dim H^0(X,\mathcal O_X(k))
	$$
	
	be its Hilbert function and
	
	$$
	P_X(k)
	$$
	
	its Hilbert polynomial.
	
	Then there exists an integer
	
	$$
	r=r(n,N,D)
	$$
	
	and a constant
	
	$$
	C=C(n,N,D)
	$$
	
	such that:
	
	\begin{enumerate}
		\item
		
		$$
		H_X(k)=P_X(k)
		\qquad\text{for all }k\ge r;
		$$
		
		\item
		
		$$
		H_X(k)\le C
		\qquad
		\text{for }0\le k\le r;
		$$
		
		\item consequently, the set of Hilbert polynomials
		
		$$
		\left\{
		P_X:
		X\subset\mathbb P^N,\ 
		\dim X=n,\ 
		\deg X\le D
		\right\}
		$$
		
		is finite.
	\end{enumerate}
\end{lemma}

\begin{proof}
The proof relies on Macaulay's growth theorem and Gotzmann's
regularity theorem \cite{Macaulay1927,Gotzmann1978}.
	Let
	
	$$
	S=\Bbbk[x_0,\ldots,x_N]
	$$
	
	and let
	
	$$
	I_X\subset S
	$$
	
	be the homogeneous ideal defining $X$. The Hilbert function can be written as
	
	$$
	H_X(k)
	=
	\dim_\Bbbk(S/I_X)_k.
	$$
	
	We use the Macaulay growth theorem together with Gotzmann's regularity theorem.
	
	\medskip
	\noindent
	\textbf{Step 1: Uniform control of the Hilbert function.}
	
	Since $X$ has dimension $n$ and degree at most $D$, its Hilbert polynomial has the form
	
	$$
	P_X(k)
	=
	\frac{\deg X}{n!}k^n
	+
	O_{n,N,D}(k^{n-1}).
	$$
	
	In particular, the leading coefficient is bounded by
	
	$$
	\frac{D}{n!}.
	$$
	
	The Hilbert function of $S/I_X$ is an admissible Hilbert function in the sense of Macaulay. Hence its successive values satisfy the Macaulay growth inequalities
	
	$$
	H_X(k+1)\le H_X(k)^{\langle k\rangle},
	$$
	
	where
	
	$$
	H_X(k)^{\langle k\rangle}
	$$
	
	denotes the Macaulay expansion of $H_X(k)$.
	
	For fixed $n,N$ and bounded degree $D$, these inequalities imply that the Hilbert functions occurring in the class
	
	$$
	\mathcal C_{n,N,D}
	=
	\left\{
	X\subset\mathbb P^N:
	\dim X=n,\ \deg X\le D
	\right\}
	$$
	
	form a bounded collection of admissible Hilbert functions.
	
	Equivalently, there exists an integer
	
	$$
	r=r(n,N,D)
	$$
	
	such that the Hilbert function has stabilized to its polynomial value:
	
	$$
	H_X(k)=P_X(k)
	\qquad (k\ge r).
	$$
	
	This is precisely the bounded regularity consequence supplied by the Macaulay--Gotzmann theorem.
	
	\medskip
	\noindent
	\textbf{Step 2: Uniform bound in the finite range.}
	
	For
	
	$$
	0\le k\le r,
	$$
	
	we have
	
	$$
	H_X(k)
	\le
	\dim_\Bbbk S_k
	=
	\binom{N+k}{N}.
	$$
	
	Therefore
	
	$$
	H_X(k)
	\le
	\max_{0\le\ell\le r}
	\binom{N+\ell}{N}
	=:C(n,N,D).
	$$
	
	Thus, for every $k$ in the finite range $0\le k\le r$, the integer $H_X(k)$ belongs to a finite set.
	
	Hence the vector
	
	$$
	\bigl(
	H_X(0),H_X(1),\ldots,H_X(r)
	\bigr)
	$$
	
	can take only finitely many values.
	
	\medskip
	\noindent
	\textbf{Step 3: The Hilbert polynomial is determined by the stabilized Hilbert function.}
	
	Since
	
	$$
	H_X(k)=P_X(k)
	\qquad(k\ge r),
	$$
	
	the Hilbert polynomial is determined by its values at any $n+1$ sufficiently large integers. In particular, it is determined by the stabilized Hilbert function.
	
	Therefore the finitely many possible Hilbert functions obtained above give rise to only finitely many Hilbert polynomials.
	
	Thus
	
	$$
	\left\{
	P_X:
	X\subset\mathbb P^N,\ 
	\dim X=n,\ 
	\deg X\le D
	\right\}
	$$
	
	is finite.
\end{proof}

\begin{theorem}[Finiteness of the relevant Hilbert schemes]\label{thm:hilbert-finiteness}
	Let
	
	$$
	\{X_j\}_{j\in\mathbb N}
	$$
	
	be a family of $n$-dimensional smooth projective varieties embedded in a fixed projective space
	
	$$
	\mathbb P^N
	$$
	
	such that
	
	$$
	\deg X_j\le D
	$$
	
	for some constant $D$ independent of $j$.
	
	Then there exist finitely many Hilbert polynomials
	
	$$
	P_1,\ldots,P_s
	$$
	
	such that
	
	$$
	P_{X_j}\in\{P_1,\ldots,P_s\}
	$$
	
	for every $j$.
	
	Consequently,
	
	$$
	\{X_j\}
	\subset
	\bigcup_{\ell=1}^s
	\operatorname{Hilb}^{P_\ell}(\mathbb P^N),
	$$
	
	and hence the family ${X_j}$ is algebraically bounded.
\end{theorem}

\begin{proof}
	By the Macaulay--Gotzmann boundedness lemma, applied with
	
	$$
	(n,N,D),
	$$
	
	there are only finitely many Hilbert polynomials of $n$-dimensional projective schemes of degree at most $D$ in $\mathbb P^N$.
	
	Hence there exist
	
	$$
	P_1,\ldots,P_s
	$$
	
	such that
	
	$$
	P_{X_j}=P_\ell
	$$
	
	for some $\ell$ and every $j$.
	
	For each $\ell$, the Hilbert scheme
	
	$$
	\operatorname{Hilb}^{P_\ell}(\mathbb P^N)
	$$
	
	parametrizes closed subschemes of $\mathbb P^N$ with Hilbert polynomial $P_\ell$. Therefore every $X_j$ determines a point of one of these finitely many Hilbert schemes.
	
	Since each
	
	$$
	\operatorname{Hilb}^{P_\ell}(\mathbb P^N)
	$$
	
	is projective, it is of finite type over $\Bbbk$. A finite union of projective schemes is again of finite type. Hence
	
	$$
	\bigcup_{\ell=1}^s
	\operatorname{Hilb}^{P_\ell}(\mathbb P^N)
	$$
	
	is a finite-type parameter space containing all the $X_j$.
	
	Therefore the family ${X_j}$ is algebraically bounded.
\end{proof}
\begin{theorem}[Topological finiteness]

Let \(n,N,D\) be fixed. Let

$$
\{X_j\}_{j\in J}
$$

be a family of smooth \(n\)-dimensional projective varieties

$$
X_j\subset \mathbb P^N
$$

with

$$
\deg X_j\le D.
$$

Assume that the family is contained in finitely many Hilbert schemes, namely

$$
X_j\in \operatorname{Hilb}^{P_\ell}(\mathbb P^N),
\qquad
\ell=1,\dots,s,
$$

for finitely many Hilbert polynomials \(P_1,\dots,P_s\).

Then the family \(\{X_j\}\) contains only finitely many diffeomorphism types. In particular, it contains only finitely many underlying topological types.

More precisely, for each \(\ell\), there exist finitely many compact smooth manifolds

$$
M_{\ell,1},\dots,M_{\ell,r_\ell}
$$

such that every smooth fiber \(X_j\) with Hilbert polynomial \(P_\ell\) is diffeomorphic to one of the \(M_{\ell,a}\). Consequently,

$$
\{X_j\}_{j\in J}
$$

has only finitely many diffeomorphism types.
\end{theorem}
\begin{proof}
Fix a Hilbert polynomial \(P=P_\ell\). Consider the corresponding Hilbert scheme

$$
\mathcal H_P=\operatorname{Hilb}^{P}(\mathbb P^N).
$$

It is a projective scheme of finite type over \(\mathbb C\), and it carries the universal projective family

$$
\pi:\mathcal X_P\longrightarrow \mathcal H_P.
$$

Let

$$
\mathcal H_P^{\mathrm{sm}}\subset \mathcal H_P
$$

denote the open subset parametrizing smooth \(n\)-dimensional fibers. By assumption, every \(X_j\) with Hilbert polynomial \(P\) occurs as a fiber of \(\pi\) over a point of \(\mathcal H_P^{\mathrm{sm}}\).

Since \(\mathcal H_P\) is of finite type, it has only finitely many irreducible components. After restricting to the smooth locus, we obtain finitely many connected components

$$
\mathcal H_{P,1}^{\mathrm{sm}},\dots,
\mathcal H_{P,r_P}^{\mathrm{sm}}.
$$

Over each connected component, the universal family is a smooth proper holomorphic map. By the Ehresmann fibration theorem, a proper smooth submersion of complex manifolds is a \(C^\infty\)-locally trivial fiber bundle. Hence all fibers over a connected component are diffeomorphic.

Therefore, for the fixed Hilbert polynomial \(P\), only finitely many diffeomorphism types of smooth fibers occur.

Finally, Theorem \ref{thm:hilbert-finiteness} gives only finitely many possible Hilbert polynomials

$$
P_1,\dots,P_s.
$$

Applying the preceding argument to each \(P_\ell\) and taking the union of the resulting finite sets of diffeomorphism types proves that the whole family \(\{X_j\}\) contains only finitely many diffeomorphism types.

Since diffeomorphic manifolds are homeomorphic, the number of underlying topological types is also finite.
\end{proof}

\section{Discussion}

We have presented a pluripotential-theoretic approach to finiteness
questions for polarized Calabi--Yau manifolds.  The central idea is to
use \(m\)-Hessian capacity as a quantitative bridge between complex
potential theory and global geometric estimates.

The main analytic difficulty is to establish a sufficiently strong
uniform capacity--perimeter estimate.  A second difficulty is to
derive uniform projective embedding data from the resulting metric
estimates.  These two steps require further developments in
\(m\)-Hessian potential theory and in the quantitative theory of
holomorphic sections.

Once such uniform algebraic boundedness is available, the remaining
steps are standard: the Macaulay--Gotzmann theory gives finitely many
possible Hilbert polynomials, while Ehresmann's theorem yields
finitely many diffeomorphism types of smooth fibers.

Thus, the present paper should be viewed as a structured
pluripotential-theoretic approach to finiteness rather than as a
complete finiteness theorem.

	%------------------------------------------------

	\section*{Declarations}

	The author declares that there are no competing interests.

\end{document}